\documentclass[12pt, a4paper]{article}
\usepackage{amsfonts,amssymb,amsmath,amscd,latexsym,makeidx,theorem,graphics}

\usepackage{hyperref}
\usepackage{color}

\usepackage[pdftex]{graphicx}

\newtheorem{thm}{Theorem}[section]
\newtheorem{thmx}{Theorem}[section]

\newtheorem{lem}[thm]{Lemma}

\newtheorem{cor}[thm]{Corollary}

\newtheorem{rem}{Remark} 

\newenvironment{proof}{\noindent\emph{Proof.}}{\hfill$\square$\medskip}

\newcommand{\D}{\Delta}

\newcommand{\vp}{\varphi}
\newcommand{\R}{\mathbb{R}}
\newcommand{\ve}{\varepsilon}

\newcommand{\K}{\mathcal{K}}

\DeclareMathOperator{\loc}{loc}

\author{Ali Hyder\thanks{The author is supported by the Swiss National Science Foundation projects no.  P2BSP2-172064 and P400P2-183866.}\\ {\small Johns Hopkins University}\\ {\small \texttt{ahyder4@jhu.edu} } }

\title{Concentration phenomena to a higher order Liouville equation} 

\begin{document}

\maketitle

\abstract 

We study blow-up and quantization phenomena for a sequence of solutions $(u_k)$ to the prescribed $Q$-curvature problem $$  (-\Delta)^nu_k= Q_ke^{2nu_k}\quad \text{in }\Omega\subset\R^{2n},\quad \int_{\Omega}e^{2nu_k}dx\leq C,$$   under natural assumptions on $Q_k$.  It is well-known that, up to a subsequence, either $(u_k)$ is bounded in a suitable norm, or there exists $\beta_k\to\infty$ such that $ u_k=\beta_k(\vp+o(1))$ in $\Omega\setminus (S_1\cup S_\vp)$ for some non-trivial non-positive $n$-harmonic function $\vp$ and for a finite set $S_1$, where  $S_\vp$ is the zero set  of $\vp$.  We  prove quantization of the total curvature $\int_{\tilde\Omega}Q_ke^{2nu_k}dx$ on the region $\tilde\Omega\Subset(\Omega\setminus S_\vp)$. We also consider a  non-local case in dimension three.

 \section{Introduction to the problem}

Given a bounded domain $\Omega\subset\R^{2n}$, we will consider a sequence $(u_k)$ of solutions to the prescribed $Q$-curvature equation
\begin{align}\label{eq-1}
   (-\D)^nu_k= Q_ke^{2nu_k}\quad \text{in }\Omega,
\end{align}
under the uniform (volume) bound
\begin{align}\label{eq-2}
\int_{\Omega}e^{2nu_k}dx\leq C,\quad k=1,2,3,\dots
\end{align}
and suitable bounds on $Q_k\geq 0$.

Contrary to the two dimensional situation studied by Br\'ezis-Merle \cite{BM} and  Li-Shafrir \cite{LS} (see also \cite{DR,   Lin-Wei,  Mal, ndi, RW}) where blow up occurs only on a finite set $S_1$, in dimension $4$ and higher it is possible to have blow-up on larger sets. More precisely, for a finite set $S\subset\Omega$ let us introduce
\begin{equation}\label{defK}
\mathcal{K}(\Omega, S):=\{\varphi\in C^\infty(\Omega \setminus S):\varphi\le 0,\,\varphi\not\equiv 0,\, \Delta^n \varphi\equiv 0\},
\end{equation}
and for a function $\varphi \in \mathcal{K}(\Omega,S)$ set
\begin{equation}\label{defS0}
S_\varphi:=\{x\in\Omega\setminus S: \varphi(x)=0\}.
\end{equation} We shall use the notations $S_\infty$ and $S_{sph}$ to denote the set of all \emph{blow-up} points and the set of all \emph{spherical blow-up} points respectively, where such points are defined as follows: 

A point $x\in \Omega$ is said to be a \emph{blow-up} point  
if there exists a sequence of points $(x_k)$ in $\Omega$ such that $$x_k\to x\quad\text{and  }u_k(x_k)\to\infty.$$  A point  $x\in S_\infty$ is said to be a \emph{spherical blow-up}  point if there exists $x_k\to x$ and $r_k\to 0^+$ such that for some $c\in \R$ 
  $$\eta_k(x):=u_k(x_k+r_kx)+\log r_k+c\to\eta (x) \quad\text{in }C^{2n-1}_{loc}(\R^{2n}),$$    where $\eta$ is a spherical solution to 
\begin{align}\label{eq-R2n}
(-\D)^\frac m2 u=(m-1)!e^{mu}\quad\text{in }\R^{m},\quad \int_{\R^{m}} e^{mu(x)}dx<\infty,
\end{align} with $m=2n$, 
that is, $\eta$ is of the form \begin{align}\label{spherical}
\eta(x)=\log\left(\frac{2\lambda}{1+\lambda^2|x-x_0|^2}\right),
\end{align}    $\text{for some }\lambda>0 \text{ and }x_0\in\R^{m}.$

\begin{thmx}[\cite{ARS},  \cite{Mar1}]\label{ARSM}   
Let $\Omega$ be a bounded domain in $\R^{2n}$, $n>1$ and let $(u_k)$ be a sequence of solutions to \eqref{eq-1}-\eqref{eq-2}, where 
\begin{align}\label{Q0}
Q_k\rightarrow Q_0>0 \quad\text{in }C^0_{loc}(\Omega),
\end{align}
 and define the set   
$$S_1:=\left\{x\in\Omega:\lim_{r\to0^+}\liminf_{k\to\infty}\int_{B_r(x)}Q_ke^{2nu_k}dx\geq \frac{\Lambda_1}{2}\right\},\quad \Lambda_1:=(2n-1)!|S^{2n}|.$$ 
Then up to extracting a subsequence one of the following is true.
\begin{itemize}
 \item [i)] For every $0\leq \alpha<1$, $(u_k)$  is bounded in   $C^{2n-1,\alpha}_{\loc}(\Omega)$.
 \item [ii)] There exists $\vp\in \mathcal{K}(\Omega, S_1)$  and a sequence of numbers $\beta_k\to\infty$ such that 
 \begin{equation}\label{convbetak} \frac{u_k}{\beta_k}\to\vp\quad \text{in } C_{\loc}^{2n-1,\alpha}(\Omega\setminus (S_1\cup S_\varphi)),  \quad 0\leq \alpha<1. \end{equation} In particular $u_k\to-\infty$ locally uniformly in $\Omega\setminus (S_1\cup S_\varphi)$.  \end{itemize} \end{thmx}
 
Notice that Theorem \ref{ARSM} contains the results of \cite{BM}   since when $ n= 1$ by maximum principle we have $S_\vp=\emptyset$ for every $\vp\in \K(\Omega,S_1)$.   In fact the more complex blow-up behavior   for $n>1$ can be seen as a consequence of the size of $\K(\Omega, S_1)$. A way of recovering  the finiteness of the blow-up set   $S_\infty$ was given by Robert \cite{Rob2}  for  $n = 2$ and generalized by Martinazzi \cite{Mar2} for $n \geq 3$: 
 
 \begin{thmx}[\cite{Mar2, Rob2}]\label{RM} 
 Let $(u_k)$ be a sequence of solutions to \eqref{eq-1}, \eqref{eq-2} and \eqref{Q0}. Assume that we are in case $ii)$ of Theorem \ref{ARSM}.  If  $(u_k)$ satisfies   \begin{align}\label{cond-uk}  \int_{B_\rho(\xi)}|\D u_k|dx\leq C,\quad \int_{\Omega}(\D u_k)^-dx\leq C,\quad (\D u_k)^-:=\max\{-\D u_k,0\},  \end{align}  for some $B_\rho(\xi)\subset \Omega$,   then  $S_1=\{x^{(1)},x^{(2)},\dots, x^{(M)}\}$ is a finite set and    $$Q_ke^{2nu_k}\rightharpoonup \sum_{i=1}^M N_i\Lambda_1\delta_{x^{(i)}}$$ in the sense of measures, where $N_i\in\mathbb N:=\{1,2,3,\dots\}$.  Moreover, $u_k\to-\infty $ locally uniformly in $\Omega\setminus S_1$.   \end{thmx}

It is worth pointing out that  without the additional assumption \eqref{cond-uk} the blow-up set $S_\infty$ need not be finite. In general, $S_\infty$ could be a hypersurface, see e.g. \cite{ARS, HIM, HM-gluing}. \\

In our first theorem  we show that  the profile of $\frac{u_k}{\beta_k}$ near the zero set $S_\vp$ is very closed to that of $\vp$ in $C^{1,\alpha}$ norm,  and we give a characterization of the blow-up points outside the zero set $S_\vp$.   
  
\begin{thm}\label{thm1}
Let $\Omega$ be a bounded domain in $\R^{2n}$, $n>1$ and let $(u_k)$ be a sequence of solutions to \eqref{eq-1}-\eqref{eq-2} for some $Q_k$ satisfying \eqref{Q0}.  Assume that we are in case $ii)$ of Theorem \ref{ARSM}. 
Then   
\begin{align}\label{conv-1}
\vp\in \mathcal{K}(\Omega, \emptyset) \quad\text{and }  \frac{u_k}{\beta_k}\to\vp\quad \text{in } C_{\loc}^{2n-1,\alpha}(\Omega\setminus (S_1\cup S_\vp)),  \quad 0\leq \alpha<1, \end{align}
\begin{equation}\label{conv-2}
\frac{u_k}{\beta_k}\to\vp\quad \text{in } C_{\loc}^{1,\alpha}(\Omega\setminus S_{sph}),\quad 0\leq\alpha<1.
\end{equation}
In addition, if $Q_k$ is bounded in $ C^1_{loc}(\Omega)$  then  \begin{align}\label{critical}
S_\infty\subseteq  \{x\in\Omega:\nabla\vp(x)=0\}. 
\end{align}
 \end{thm}
 
 We remark that the $C^{1,\alpha}$ convergence in \eqref{conv-2} is sharp in the sense that  it is not  true in $C^2_{\loc}(\Omega\setminus S_{sph})$, see Example 1 in subsection \ref{Examples}.

As an immediate consequence of \eqref{conv-1}-\eqref{conv-2} we prove the following: 
\begin{cor}\label{cor} Let $(u_k)$ be a sequence of solutions to  \eqref{eq-1}, \eqref{eq-2} and \eqref{Q0}. Then \begin{itemize}\item[i)] $\dot S_{sph}:=S_\infty\setminus S_\vp\subseteq S_{sph}\quad\text{and }S_\infty\setminus S_{sph}\subseteq S_\vp.  $ \item[ii)] If the scalar curvature $R_{g_{u_k}}$ of the  conformal metric $g_{u_k}:=e^{2u_k}|dx|^2$ is uniformly bounded from below, then $S_\infty=S_{sph}$.  \end{itemize}
\end{cor}

As we have already mentioned that the quantization result of Theorem \ref{RM} is not true without the additional assumption \eqref{cond-uk}. However, it turns out that  a quantization result still holds  if we stay outside the  zero set $S_\vp$.  More precisely, we have:
\begin{thm}\label{thm2} Let $(u_k)$ be a sequence of solutions to \eqref{eq-1}, \eqref{eq-2} and \eqref{Q0}.      Then the set    $\dot S_{sph}:=S_\infty\setminus S_\vp$ is finite, and  up to a subsequence,  \begin{align}\label{NLambda} Q_ke^{2nu_k} \rightharpoonup\sum_{x^\ell\in\dot S_{sph}}N_\ell\Lambda_1\delta_{x^{(\ell)}}\quad\text{in }\Omega\setminus S_\vp,\end{align} in the sense of measures, where  $N_\ell\in\mathbb{N}$, and the above sum is considered to be $0$ if $\dot S_{sph}=\emptyset$.   Moreover, if $(Q_k)$ is bounded in $C^1_{loc}(\Omega)$   then for every $x^\ell\in \dot S_{sph}$  we have \begin{align}\label{uk=betak} u_k(x_k^{(\ell)}):= \max_{B_{\delta_0}(x^{(\ell)})}u_k=-\beta_k(\vp(x^{(\ell)})+o(1)),\quad\quad o(1)\xrightarrow{k\to\infty}0,\end{align}   for some  $\delta_0>0$.  
Finally, if  the Hessian $\nabla^2\vp(x^{(\ell)})$ is strictly positive definite for some $x^\ell\in \dot S_{sph}$, and  $Q_k\equiv 1$   then for $k$ large \begin{align}\label{vol-lower} \int_{B_{\delta_0}(x^{(\ell)})}e^{2nu_k}dx \geq \Lambda_1+c_0u_k(x_k^{(\ell)})e^{-2u_k(x_k^{(\ell)})},\end{align} for some $c_0>0$. 
\end{thm}

Theorem \ref{thm2} contains Theorem \ref{RM} as \eqref{cond-uk} implies that $S_\vp=\emptyset$. Our first proof of \eqref{NLambda} is based on  Theorem \ref{RM}. Under an additional assumption, namely  $(Q_k)$ is bounded in $C^1_{loc}(\Omega)$,  we give a direct  proof (without using Theorem \ref{RM}). In this case we also derive a lower bound of the distances  between the locations of   ``peaks'' at a   spherical blow-up point $x^\ell$ with $N_\ell>1$, see Lemma \ref{gap}

In our next theorem we give a sufficient condition on the poly-harmonic function $\vp$ to rule out  the possibility of collapsing  multiple spherical bubbles at a blow-up point in $S_\infty\setminus S_{\vp}$.  

\begin{thm}\label{thm3} Let $(u_k)$ be a sequence of solutions to \eqref{eq-1}, \eqref{eq-2} and \eqref{Q0}. Assume that $(Q_k)$ is bounded in $C^1_{loc}(\Omega)$.   If the Hessian $\nabla^2\vp(x_0)$ is strictly negative definite for some $x_0\in S_\infty$, then there exists $\delta>0$  such that  
$$ \limsup_{k\to\infty}\int_{B_\delta(x_0)}Q_ke^{2nu_k}dx\leq \Lambda_1.$$ In particular, if $x_0\in S_\infty\setminus S_\vp$, then $$ \limsup_{k\to\infty}\int_{B_\delta(x_0)}Q_ke^{2nu_k}dx= \Lambda_1.$$
\end{thm}

The assumption on $\nabla^2\vp(x_0)$ in Theorem \ref{thm3} is necessary as it is not  true if  $\nabla^2 \vp(x_0)=0$, see Examples 2, 3 in subsection \ref{Examples}.   

\medskip 

Now we move on to the non-local case in dimension $3$.  More precisely, we shall consider the non-local equation $(-\D)^\frac 32u_k=Q_ke^{3u_k}$ on a bounded domain  $\Omega\subset\R^3$.  The operator $(-\D)^\frac32$ can be understood as a Dirichlet-to-Neuman map via bi-harmonic extension on the upper-half space $\R^4_+$. For a precise definition, and notations see Section \ref{section-nonlocal}.  

\begin{thmx}[\cite{DGHM}] \label{thmC} Let $\Omega$ be a smooth bounded domain in $\R^3$.   Let $(U_k)\subset C^0(\overline{\R^4_+})$ be a sequence of representable  solutions to 
 \begin{equation}\label{eq-1ext}\left\{ \begin{split}
  \D^ 2 U_k=0 \quad \text{ in }\R^4_+,\\
   \partial_t U_k=0 \quad \text{ in }\R^3,\\
    \mathcal L_{\frac{3}{2}}U_k=Q_ke^{3u_k} \quad\text{on }  \Omega,
  \end{split}\right.
 \end{equation}
where $u_k:=U_k|_{t=0}$ is the boundary data and $Q_k\in C^0(\overline{\Omega})$ is uniformly bounded in $L^\infty(\Omega)$.
We assume that
\begin{align}\label{cond-vol}
 \int_{\Omega}e^{3u_k}\,dx\leq C.
\end{align}
and
\begin{equation}\label{extra-assumption}
\int_{\R^3}\frac{u_k^+(x)}{1+|x|^6}\,dx\leq C.
\end{equation}
Set
$$S_1:=\left\{ x\in\Omega:\lim_{\ve\to 0^+}\liminf_{k\to\infty}\int_{B_\ve(\bar x)}|Q_k|e^{3u_k}\,dx\geq \frac{\Lambda_1}{2} \right\},\quad \Lambda_1=2| S^3|=4\pi^2.$$
Then $S_1$ is a finite set and, up to a subsequence, one of the following is true:
\begin{itemize}
 \item [(i)] $U_k\to U_\infty$ in $C^{2,\alpha}_{loc}(\R^4_+\cup (\Omega\setminus S_1))$ for any $\alpha\in [0,1)$,

 \item [(ii)] There exists $\Phi\in \K(\Omega)$ and numbers $\beta_k\to\infty$ such that
 $$\frac{U_k}{\beta_k}\to\Phi\quad\text{in }C^{2,\alpha}_{\loc}((\R^4_+\cup \Omega)\setminus S),\quad S=S_\Phi \cup S_1,$$
 where
$S_\Phi:=\{x\in \Omega:\Phi((x,0))=0\}$.
Moreover $S_\Phi$ has dimension at most $2$.
\end{itemize}
\end{thmx}

We improve Theorem \ref{thmC} by showing the following theorem:

\begin{thm}\label{thm-nonlocal} Let $\Omega$ be a smooth bounded domain in $\R^3 $.   Let $(U_k)\subset C^0(\overline{\R^4_+})$ be a sequence of representable  solutions to \eqref{eq-1ext}, \eqref{cond-vol} and \eqref{extra-assumption}  for some $$Q_k\to Q_0>0\quad\text{in }C^0(\bar\Omega),\quad \|\nabla Q_k\|_{C^0(\bar \Omega)}\leq C.$$   Assume that we are in case $ii)$ of Theorem \ref{thmC}. Then there is a finite set (possibly empty) $\{x^1,\dots, x^N\}\subset S_1\cap (\Omega\setminus S_\Phi)$ such that    \begin{align}\label{conv-23}  \frac{U_k}{\beta_k}\to\Phi\quad\text{in }C^{1,\alpha}_{loc} ((\R^4_+\cup \Omega)\setminus \{x^1,\dots,x^N\}),\quad 0\leq \alpha<1,\end{align} and  \begin{align}\label{quant-24} Q_ke^{3u_k}dx  \rightharpoonup\sum_{\ell=1}^NN_\ell 4\pi^2\delta_{x^{(\ell)}}\quad\text{in }\Omega\setminus S_\Phi,\end{align} in the sense of measures, where  $N_\ell\in\mathbb{N}$.    \end{thm}

 We shall  use   the following classification result from \cite{JMMX, Lin, Mar0}, see also \cite{HMM, Xu} and the references therein. 

\begin{thmx} \label{thmclas}
Let $u$ be a solution to \eqref{eq-R2n} with $m\geq 2$. Then either $u$ is  of the form \eqref{spherical} for some  $x_0\in\R^{m}$ and $\lambda>0$, or $u=v+p$ where $p$ is  an upper bounded non-constant polynomial of degree at most  $m-1$ and  $v(x)=O(\log (2+|x|))$. Moreover, if $u$ satisfies $\int_{B_R}|\D u|dx\leq CR^{m-2}$ for every $R\geq 1$ (equivalently, $p\equiv const$)  then $u$ is of the form \eqref{spherical}.
\end{thmx}
 
\section{Proof of Theorems \ref{thm1}, \ref{thm3} and Corollary \ref{cor}}

We begin with the proof of   \eqref{conv-1}.  

\medskip 

\noindent {\bf\emph{Proof of  \eqref{conv-1}}}  The proof is similar to the one in  \cite{ARS}, see also \cite{Mar1}. The crucial  difference is that we need to show that the function $\vp$ is smooth in $\Omega$. For that purpose we decomposition  $u_k$ on the whole domain $\Omega$ instead of a small neighborhood of a point in $\Omega\setminus S_1$ as done in \cite{ARS, Mar1}.  More precisely, we write  $u_k=v_k+h_k$ where $\D^n h_k\equiv 0$ in $\Omega$ and  $v_k$ is given by   \begin{align}\label{def-vk}
v_k(x):=\frac{1}{\gamma_n}\int_{\Omega}\log\left(\frac{1}{|x-y|} \right)Q_k(y)e^{2nu_k(y)}dy. 
\end{align} 
Here the constant  $\gamma_n:=\frac{(2n-1)!}{2}|S^{2n}|$ satisfies $$(-\D)^n\log\frac{1}{|x|}=\gamma_n\delta_0\quad\text{in }\R^{2n}.$$

If necessary, restricting ourselves in a smaller domain $\tilde \Omega\Subset \Omega$, we can assume that $$Q_k\to Q_0\quad\text{ in }C^0(\bar \Omega). $$
It follows from \eqref{def-vk} and \eqref{eq-2} that $$\int_{\Omega}|v_k|dx\leq \frac{\|Q_k\|_{L^\infty(\Omega)}}{\gamma_n}\int_{\Omega}e^{2nu_k(y)}\int_{\Omega}|\log|x-y||dxdy\leq C,$$ and hence, again by \eqref{eq-2}, we have $$\int_{\Omega}h_k^+dx\leq \int_{\Omega}(u_k^++|v_k|)dx\leq \int_{\Omega}(e^{2nu_k}+|v_k|)dx\leq C.$$  
Assume that there exists $x_0\in\Omega$ and $R_0>0$ such that ($\beta_k\leq C$ corresponds to the case $i)$ of Theorem \ref{ARSM}) $$\beta_k:=\int_{B_{R_0}(x_0)}|h_k|dx\xrightarrow{k\to\infty}\infty,\quad B_{2R_0}(x_0)\subset\Omega.$$ Then the function $\vp_k:=\frac{h_k}{\beta_k}$ satisfies $$\D^n\vp_k=0\quad\text{in } \Omega,\quad \int_{B_{R_0}(x_0)}|\vp_k|dx=1,\quad \limsup_{k\to\infty}\int_{\Omega}\vp_k^+dx=0.$$  By Lemma \ref{polyharmonic},   up to a subsequence, we have for every $\ell\in\mathbb{N}$  
 \begin{align}\label{conv-vp}\frac{h_k}{\beta_k}=\vp_k\to\vp\quad\text{in }C^{\ell}_{loc}(\Omega).\end{align} It follows that  $$\D^n\vp=0\quad\text{in }\Omega,\quad \int_{B_{R_0}(x_0)}|\vp|dx=1,\quad \int_{\Omega}\vp^+dx=0,$$ and hence, $\vp\in \K(\Omega,\emptyset)$.  
 
We claim that   $v_k$ is bounded in $C^{2n-1,\alpha}_{loc}(\Omega\setminus S)$ where $S:=S_1\cup S_\vp$. In order to prove the claim first we fix a point $\xi\in \Omega\setminus S$. Then there exists $R>0$ such that $B_{2R}(\xi)\subset \Omega\setminus S$ and up to extracting a subsequence $$\limsup_{k\to\infty}\int_{B_{2R}(\xi)}Q_ke^{2nu_k}dx=:\Lambda_{R,x_0}<\frac{\Lambda_1}{2}=\gamma_n.$$ For $x\in B_R(\xi)$ we bound 
\begin{align}\label{vk-1}
|v_k(x)|\leq \frac{1}{\gamma_n}\int_{B_{2R}(\xi)}\left| \log|x-y|\right|Q_k(y)e^{2nu_k(y)}dy+C,  
\end{align}
where we have used \eqref{eq-2} and  $|\log |x-y||\leq C$ for $(x,y)\in B_R(\xi)\times \left( \Omega\cap B^c_{2R}(\xi)\right)$. Choosing $p>1$ such that $${p}\Lambda_{R,\xi}<{\gamma_n},$$ and together with Jensens inequality we obtain from \eqref{vk-1} that 
\begin{align*}
\int_{B_R(\xi)}e^{2np|v_k|(x)}dx\leq C\int_{B_R(\xi)}\int_{B_{2R}(\xi)}\frac{f_k(y)}{\|f_k\|} e^{q_k|\log x-y| |}dydx\leq C, 
\end{align*}
where  $$f_k:=Q_ke^{2nu_k},\quad \|f_k\|:=\|Q_ke^{2nu_k}\|_{L^1(B_{2R}(\xi))}, \quad\text{and } q_k:=\frac{2np}{\gamma_n}\|f_k\|,\quad  \limsup_{k\to\infty} q_k<2n.$$   This shows that  $$\int_{B_R(\xi)}e^{2npu_k(x)}dx\leq \int_{B_R(\xi)}e^{2npv_k(x)}dx\leq C,$$ as  $h_k<0$ on $B_R(\xi)$. By H\"older  inequality, from \eqref{vk-1} one gets $\|v_k\|_{L^\infty(B_\frac R2(\xi))}\leq C$, which implies that $e^{2nu_k}\leq C$ on $B_{\frac R2}(\xi)$. Hence, $v_k$ is bounded in $C^{2n-1,\alpha}(B_\frac R4(\xi))$, and our claim follows immediately by a covering argument.  

This finishes the proof of    \eqref{conv-1}. 
\hfill $\square$

\medskip

\subsection{Proof of \eqref{conv-2}} 
 \medskip 
 
Let us first introduce some notations. 
For a sequence of points $(x_k)$ in $ \Omega$ and a sequence of positive numbers $(\mu_k)$  with $B_{\mu_k}(x_k)\subset\Omega$ we let $s_k\in [0,\mu_k]$ and $\bar x_k\in \bar B_{s_k}(x_k)$ be such that (compare  \cite{ARS}) 
 \begin{align}\label{def-maxima}
(\mu_k-s_k)e^{u_k(\bar x_k)}=(\mu_k-s_k)\max_{ \bar B_{s_k}(x_k)}e^{u_k}  =\max_{r\in [0,\mu_k]}\left((\mu_k-r)\max_{ \bar B_{r}(x_k)}e^{u_k} \right)=:L_k.
\end{align}
Note that  $|x_k-\bar x_k|=s_k$,  $u_k\leq u_k(\bar x_k)$ in $\bar B_{s_k}(x_k)$, and    \begin{align}\label{uperest}  L_k\geq \max_{r\in [0,\frac{\mu_k}{2}]}\left((\mu_k-r)\max_{x\in \bar B_{r}(x_k)}e^{u_k(x)} \right)\geq \frac{\mu_k}{2} \max_{\bar B_\frac{\mu_k}{2}(x_k)}e^{u_k}.\end{align}
Setting   \begin{align}\label{def-rk}
r_k:=\frac{\mu_k-s_k}{L_k},
\end{align}
one gets  \begin{align}\label{uper}u_k(\bar x_k+r_kx)+\log r_k \leq 2\quad\text{for }|x|\leq \frac{L_k}{2}.\end{align}

The following three lemmas are crucial  in proving  \eqref{conv-2}. 

\begin{lem}\label{lem1}
 Let $\mu_k\to 0^+$ be such that $\log\mu_k=o(\beta_k)$. Assume that $$\lim_{k\to\infty}\frac{v_k(x_k)}{\beta_k}\neq 0,$$ for some $x_k\to x\in \Omega$.  Then $\lim_{k\to\infty} L_k=\infty$. 
\end{lem}
\begin{proof}
From \eqref{def-vk} one obtains 
\begin{align}\label{log}
|v_k(x_k)|&\leq C \left(\int_{\Omega\cap B^c_{\frac{\mu_k}{2}}(x_k)}+\int_{\Omega\cap B_{\frac{\mu_k}{2}}(x_k)}\right)|\log|x_k-y||e^{2nu_k(y)}dy\notag\\
&\leq C|\log \mu_k| \int_{\Omega}e^{2nu_k(y)}dy +C\left(\sup_{B_\frac{\mu_k}{2}(x_k)}e^{2nu_k}\right)\int_{ B_{\frac{\mu_k}{2}}(x_k)}|\log|x_k-y||dy\notag\\
&\leq C |\log \mu_k| \int_{\Omega}e^{2nu_k(y)}dy +C\left(\sup_{B_\frac{\mu_k}{2}(x_k)}e^{2nu_k}\right) \mu_k^{2n}|\log \mu_k|\\
&\leq C|\log\mu_k|+C\left(\frac{L_k}{\mu_k}\right)^{2n}\mu_k^{2n}|\log \mu_k|\notag\\
&\leq C|\log\mu_k|(1+L_k^{2n}),\notag
\end{align}
where  the second  last inequality follows from \eqref{eq-2} and \eqref{uperest}. Dividing the above inequality by $\beta_k$ one gets $$0\not\leftarrow \frac{v_k(x_k)}{\beta_k}=o(1)(1+L_k^{2n})\quad\Longrightarrow\quad L_k\to\infty.$$
\end{proof}

\begin{lem}\label{nabla}
 Let $\mu_k\to 0^+$ be such that $\mu_k^{-1}=o(\beta_k)$. Assume that $$\lim_{k\to\infty}\frac{|\nabla v_k(x_k)|}{\beta_k}\neq 0,$$ for some $x_k\to x\in \Omega$. Then $\lim_{k\to\infty} L_k=\infty$. 
\end{lem}
\begin{proof} 
Differentiating under the integral sign, from \eqref{def-vk}, and together with \eqref{eq-2} and \eqref{uperest} we bound 
\begin{align*}
|\nabla v_k(x_k)| 
&\leq  C\left(\int_{\Omega\cap B^c_{\frac{\mu_k}{2}}(x_k)}+\int_{\Omega\cap B_{\frac{\mu_k}{2}}(x_k)}\right)\frac{e^{2nu_k(y)}}{|x_k-y|}dy\\
&\leq \frac{C}{\mu_k}\int_{\Omega}e^{2nu_k(y)}dy+C\left( \sup_{ B_{\frac{\mu_k}{2}}(x_k)}e^{2nu_k}\right)\int_{B_{\frac{\mu_k}{2}}(x_k)}\frac{dy}{|x_k-y|}\\
&\leq \frac{C}{\mu_k}+C\left(\frac{L_k}{\mu_k}\right)^{2n}\mu_k^{2n-1}\\
&\leq  \frac{C}{\mu_k}(1+L_k^{2n}). 
\end{align*}
The lemma follows immediately. 
\end{proof}

\begin{lem}\label{nabla2}
Let $0\leq\alpha<1$ be fixed. For $i=1,2$ let $(x_{i,k})$ be  two sequences of points on $\Omega$ such that $x_{i,k}\to x_0\in\Omega$.  Assume that $x_{1,k}\neq x_{2,k}$ for every $k$ and   
$$\frac{1}{\beta_k}\frac{|\nabla v_k(x_{1,k})-\nabla v_k(x_{2,k})|}{|x_{1,k}-x_{2,k}|^\alpha}\not\to0.$$ If $\mu_{i,k}\to0^+$ satisfies $$\mu_{i,k}^2\beta_k\geq |x_{1,k}-x_{2,k}|^\frac{1-\alpha}2,$$ then $\max\{L_{1,k},L_{2,k}\}\to\infty$ ($L_{i,k}$ is defined  by taking  $x_k=x_{i,k}$ and  $\mu_k=\mu_{i,k}$ in \eqref{def-maxima}). 
\end{lem}
\begin{proof}
Differentiating under the integral sign, from \eqref{def-vk}, and using that  $$\left| \frac{x_{1,k}-y}{|x_{1,k}-y|^2}- \frac{x_{2,k}-y}{|x_{2,k}-y|^2}  \right|  
\leq |x_{1,k}-x_{2,k}|\left(\frac{1}{|x_{1,k}-y|^2}+\frac{1}{|x_{2,k}-y|^2} \right),$$  we obtain 
\begin{align*}
|\nabla v_k(x_{1,k})-\nabla v_k(x_{2,k})|&\leq C|x_{1,k}-x_{2,k}|\int_{\Omega}\left(\frac{1}{|x_{1,k}-y|^2}+\frac{1}{|x_{2,k}-y|^2} \right)e^{2nu_k(y)}dy\\
&=C|x_{1,k}-x_{2,k}|(I_1+I_2), 
\end{align*}
where $$I_i:=\int_{\Omega} \frac{1}{|x_{i,k}-y|^2} e^{2nu_k(y)}dy,\quad i=1,2.$$
As in Lemma \ref{nabla} one can show that $$I_i\leq  \frac{C}{\mu^2_{i,k}}(1+L_{i,k}^{2n}),\quad i=1,2.$$ Thus $$|\nabla v_k(x_{1,k})-\nabla v_k(x_{2,k})|\leq C|x_{1,k}-x_{2,k}|\left(\frac{1+L_{1,k}^{2n}}{\mu^2_{1,k}}+\frac{1+L_{2,k}^{2n}}{\mu^2_{2,k}}\right).$$
The lemma follows  as $\alpha<1$.
\end{proof}

\medskip 

\noindent\emph{\textbf{Proof of \eqref{conv-2}}}
Since $\frac{h_k}{\beta_k}\to\vp$ in $C^\ell_{loc}(\Omega)$ for every $\ell\in\mathbb{N}$, \eqref{conv-2} is equivalent to   \begin{align}\label{conv-vk}\frac{v_k}{\beta_k}\to0\quad\text{in }C^{1,\alpha}_{loc}(\Omega\setminus S_{sph}),\quad 0\leq\alpha<1.\end{align}
We claim that for every compact set $K\Subset \Omega\setminus S_{sph}$ we have $$\frac{1}{\beta_k}\left(\|v_k\|_{C^1(K)}  +[\nabla v_k]_{C^{0,\alpha}(K)}\right)\xrightarrow{k\to\infty}0.$$ 
We prove the claim in two steps. 

\medskip 

\noindent\textbf{Step 1} $\frac{\|v_k\|_{C^1(K)}}{\beta_k}\to 0$. 

We assume by contradiction that $\frac{\|v_k\|_{C^1(K)}}{\beta_k}\not\to 0$. Then there exists $x_k\to x_0\in \Omega\setminus S_{sph}$ such that $$\frac{|v_k(x_k)|+|\nabla v_k(x_k)|}{\beta_k}\not\to 0.$$ We set $\mu_k:=\frac{1}{\sqrt{\beta_k}}$, and let $\bar x_k$, $L_k$ and  $r_k$   be as in \eqref{def-maxima}, \eqref{def-rk}. By  Lemmas \ref{lem1}  and \ref{nabla} we have  $L_k\to\infty$,     $r_k\to 0$, and \begin{align}\label{rkbetak}r_k^2\beta_k\leq \frac{\mu_k^2}{L_k^2}\beta_k=\frac{1}{L_k^2}\to0.\end{align}
Setting  $$\eta_k(x):=u_k(\bar x_k+r_kx)+\log r_k+\frac{1}{2n}\log \frac{Q_0(x_0)}{(2n-1)!} \quad\text {for }\bar x_k+r_kx\in\Omega,$$ and by \eqref{uper} we obtain 
$$(-\D)^n\eta_k(x)=(2n-1)!\frac{\bar Q_k(x)}{Q_0(x_0)}e^{2n\eta_k(x)},\quad \eta_k\leq C\quad\text{in }B_\frac{L_k}{2}, $$ where $\bar Q_k(x):=Q_k(\bar x_k+r_kx)$. Moreover,  by \eqref{eq-2}, \eqref{def-vk}, \eqref{conv-vp} and \eqref{rkbetak}, we have for every $R>0$ 
\begin{align}\label{est-laplacian}
\int_{B_R}|\D\eta_k(x) |dx & \leq r_k^2\int_{B_R}|\D v_k(\bar x_k+r_kx)|dx +r_k^2\int_{B_R}| \D h_k(\bar x_k+r_kx) |dx \notag \\
&\leq C\int_{\Omega}Q_k(y)e^{2nu_k(y)}\int_{B_R}\frac{r_k^2dx}{|\bar x_k+r_kx-y|^2}dy+ r_k^2\beta_k O(R^{2n})\notag \\
&\leq CR^{2n-2}+o(R^{2n}).
\end{align}
 Therefore, by elliptic estimates, up to a subsequence,  $\eta_k\to\eta$ in $C^{2n-1}_{loc}(\R^{2n})$ where $\eta$ satisfies $$(-\D)^n\eta=(2n-1)!e^{2n\eta}\quad \text{in }\R^n,\quad \int_{\R^{2n}}e^{2n\eta}dx<\infty,\quad \int_{B_R}|\D\eta|dx=O(R^{2n-2}),$$ thanks to \eqref{est-laplacian}.  It follows from Theorem \ref{thmclas}  that $\eta$ is of the form \eqref{spherical}. Thus, $x_k\to x_0\in S_{sph}$, a contradiction.

\medskip 

\noindent\textbf{Step 2} $\frac{[\nabla v_k]_{C^{0,\alpha}(K)}}{\beta_k}\to 0$. 

Since  $v_k\in C^2(\Omega)$,  there exists $x_{1,k}, x_{2,k}\in K$ with $x_{1,k}\neq x_{2,k}$ such that  $$[\nabla v_k]_{C^{0,\alpha}(K)}=\sup_{x,y\in K}\frac{|\nabla v_k(x)-\nabla v_k(y)|}{|x-y|^\alpha}=\frac{|\nabla v_k(x_{1,k})-\nabla v_k(x_{2,k})|}{|x_{1,k}-x_{2,k}|^\alpha}.$$    If $|x_{1,k}-x_{2,k}|\not\to 0$ then Step 2 follows from Step 1.   
Thus, we only need to consider the case $|x_{1,k}-x_{2,k}|\to 0$. 

We assume by contradiction that   $\frac{[\nabla v_k]_{C^{0,\alpha}(K)}}{\beta_k}\not\to 0$. We set $$\mu_k=\mu_{1,k}=\mu_{2,k}:=\beta_k^{-\frac12}|x_{1,k}-x_{2,k}|^\frac{1-\alpha}{4}.$$ We let $\bar x_{i,k}$ and $L_{i,k}$ be as in    \eqref{def-maxima} with $x_k=x_{i,k}$, $i=1,2$. Then by Lemma \ref{nabla2} we get $L_k:=\max \{L_{1,k},L_{2,k}\}\to\infty$. By relabelling,  we may assume that   $L_k=L_{1,k}\to\infty$. Letting $r_k:=r_{1,k}$ (as defined in \eqref{def-rk}) we see that \eqref{rkbetak} holds. 
Now one can proceed as in Step 1 to get a contradiction.  

We conclude the proof of  \eqref{conv-2}
\hfill$\square$

\medskip 

\noindent\emph{\textbf {Proof of Corollary \ref{cor}}} It follows immediately from \eqref{conv-2} that $$S_\infty\setminus S_\vp\subseteq S_{sph}\quad\text{and }S_\infty\setminus S_{sph}\subseteq S_\vp,$$ which is the first part of the corollary. 

  We claim that $\vp\equiv const <0$ whenever 
 the scalar curvature  $R_{g_{u_k}}$  is uniformly bounded from below, that is, \begin{align}\label{scalar}R_{g_{u_k}}=-2(2n-1)e^{-2u_k}\left( \D u_k+(n-1)|\nabla u_k|^2\right)\geq -C\quad\text{in }\Omega.  \end{align}  

In order to prove the claim we fix  a ball  $B_\ve(x_0)\Subset \Omega\setminus (S_1\cup S_\vp)$. Then by \eqref{conv-1} we get $$ \D u_k+(n-1)|\nabla u_k|^2=(1+o(1))\beta_k\left (|\D \vp|+\beta_k(n-1+o(1))|\nabla \vp|^2\right)\quad\text{in }B_\ve(x_0).$$ This and \eqref{scalar} implies that $\nabla \vp\equiv0$ in $B_\ve(x_0)$, and hence $\vp\equiv const$ in $B_\ve(x_0)$. By unique continuation theorem we conclude that $\vp\equiv const$ on $\Omega$.   
In particular,  $S_\vp=\emptyset$, and from the first part of the corollary, we deduce $S_\infty=S_{sph}$. 
 \hfill $\square$
 
 \medskip

\subsection{Proof of \eqref{critical} and Theorem \ref{thm3}} 
For a given point $x_0\in\Omega$ and a constant  $\delta>0$ with  $B_{2\delta}(x_0)\subset \Omega$,  we fix a smooth  cut-off function   $\psi\in C^\infty_c(B_{2\delta}(x_0))$ such that $\psi\equiv 1$ in $B_\delta(x_0)$.  We split the function $v_k$ into $v_k=\bar v_{k}+\tilde v_{k}$ where 
 \begin{align}\label{barv}
\bar v_{k}:=\frac{1}{\gamma_n}\int_{\Omega}\log\left(\frac{1}{|x-y|}\right){\bar Q_k}(y)e^{2n\bar v_k(y)}dy,\quad {\bar Q_k }:=\psi Q_ke^{2n h_k}e^{2n\tilde v_k}.
\end{align}

\medskip

\noindent\emph{\textbf{Proof of \eqref{critical}}} Since $S_\infty\cap S_\vp\subseteq S_\vp\subseteq \{\nabla \vp=0\}$, we only need to show that $S_\infty\setminus S_\vp\subseteq \{\nabla \vp=0\}$. 

Let $x_0\in S_\infty\setminus S_\vp$.  
Then necessarily $x_0\in S_{sph}$, thanks to \eqref{conv-2}. We  choose  $\delta>0$ such that  
\begin{align}\label{19}u_k \to-\infty \quad\text{locally uniformly in } \bar B_{2\delta}(x_0)\setminus\{x_0\},\quad B_{2\delta}(x_0)\subset \Omega.\end{align} For  this choice of $\delta$ and $\psi$, we have  ${\bar Q_k}\in C_c^1(\Omega)$. Therefore,  by \eqref{poho} 
\begin{align}
0&=\int_{\Omega}\nabla {\bar Q_k}(x)e^{2n\bar v_k(x)}dx\notag\\
&=2n\int_{\Omega}\left(\nabla h_k(x) \right)\psi Q_ke^{2nu_k(x)}dx+\int_{\Omega}\left(\nabla (\psi Q_k e^{2n\tilde v_k})(x)\right) e^{2n(h_k+\bar v_k)(x)}dx\notag\\
&=:2n I_1+I_2.\label{I1+I2}
\end{align}
Since $x_0\in S_{sph}$ we have $$\lim_{\ve\to 0^+}\lim_{k\to\infty}\int_{B_\ve(x_0)}Q_k(x)e^{2nu_k(x)}dx\geq \Lambda_1,$$  which leads to      
$$\lim_{k\to\infty}\frac{|I_1|}{\beta_k}=\lim_{\ve\to 0^+}\lim_{k\to\infty} (|\nabla \vp(x_0)|+o_{\ve,k}(1)) \int_{B_\ve(x_0)}Q_k(x)e^{2nu_k(x)}\geq  \frac12\Lambda_1 |\nabla \vp(x_0)|,$$ thanks to \eqref{conv-vp} and \eqref{19}. Recalling that  $(Q_k)$ is bounded in $C^1_{loc}(\Omega)$, $\psi\equiv 1$ on $B_\delta(x_0)$, and  from  \eqref{def-vk}, \eqref{barv} and \eqref{19}, we infer \begin{align}|\nabla (\psi Q_ke^{2n\tilde v_k})|\leq Ce^{2n\tilde v_k},\label{est-31}\end{align} which gives   $|I_2|\leq C$. Plugin these estimates in \eqref{I1+I2} we obtain  $$\frac12\Lambda_1 |\nabla \vp(x_0)|\leq \lim_{k\to\infty}\frac{|I_1|}{\beta_k}=\frac{1}{2n}\lim_{k\to\infty}\frac{|I_2|}{\beta_k}=0.$$ We conclude \eqref{critical}.
\hfill $\square$

\medskip 

\noindent\emph{\textbf{Proof of Theorem \ref{thm3}}} Let $x_0\in S_\infty$ be such that  $\nabla^2\vp(x_0)$ is strictly negative definite. Then $\vp(x)<\vp(x_0)\leq 0$ for $x_0\neq x\in \bar B_{2\delta}(x_0)\subset \Omega$ for some $\delta>0$. We  can also  assume  that $\nabla^2\vp<0$ on $\bar B_{2\delta}(x_0)$ and  \eqref{19} holds. We let $x_k\in B_{2\delta}(x_0)$ be such that  $h_k(x_k):=\max _{\bar B_{2\delta }(x_0)}h_k$. As $\vp<0$ on $\bar B_{2\delta}(x_0)\setminus\{x_0\}$ we have $x_k\to x_0$ and $\nabla h_k(x_k)=0$. Therefore, for every   $ x\in B_{2\delta}(x_0)$  one has \begin{align}\label{nablahk}\nabla h_k(x)&=\int_0^1\frac{d}{dt}\nabla h_k(tx+(1-t)x_k)dt=\int_{0}^1\nabla^2h_k(tx+(1-t)x_k)[x-x_k]dt. 
\end{align}
Using that $\nabla^2\vp<0$ on $\bar B_{2\delta}(x_0)$, and by \eqref{conv-vp}, we have for $x\in  B_{2\delta}(x_0)$ \begin{align}(x-x_k)\cdot\nabla h_k(x)&=\beta_k\int_{0}^1\nabla^2\vp(tx+(1-t)x_k)[x-x_k,x-x_k]dt+o(\beta_k)|x-x_k|^2 \notag\\ 
&\leq -c_1\beta_k|x-x_k|^2, \label{est-33} \end{align} for some $c_1>0$. Now we apply Lemma \ref{poho2}  to the  integral equation \eqref{barv} with $\tilde \Omega=\Omega$, $\xi=x_k$, where $\psi \in C_c^\infty(B_{2\delta}(x_0))$ is such that $\psi=1$ on $B_\delta(x_0)$ and $\psi\geq 0$. Indeed, setting $$\lambda_k=\bar\lambda_k:=\int_{\Omega}{\bar Q_k}e^{2n\bar v_k}dx=\int_\Omega Q_k\psi e^{2nu_k}dx,$$ and by \eqref{eq-2}, \eqref{19}, \eqref{est-31} and \eqref{est-33} one obtains
\begin{align}\label{rem-39}& \frac{\lambda_k}{\Lambda_1}(\lambda_k-\Lambda_1)\notag\\ &=\frac{1}{2n}\int_\Omega (x-x_k)\cdot\nabla {\bar Q_k }(x)e^{2n\bar v_k(x)}\notag\\ &=\int_{\Omega}(x-x_k)\cdot \nabla h_k(x) \psi Q_k e^{2n u_k}dx+\frac{1}{2n}\int_{\Omega}(x-x_k)\cdot \nabla (\psi Q_k e^{2n\tilde v_k})e^{2nh_k}e^{2n\bar v_k}dx\notag \\ &\leq -c_1\beta_k\int_\Omega |x-x_k|^2\psi Q_k e^{2nu_k}dx +C\int_\Omega |x-x_k|\psi e^{2nu_k}dx\\ &\leq o(1)\notag.\end{align} Thus, $$\limsup_{k\to\infty}\int_{B_\delta(x_0)}Q_ke^{2nu_k}dx\leq \limsup_{k\to\infty}\lambda_k\leq \Lambda_1.$$   

The second part of the theorem follows from the first part and \eqref{NLambda}.
\hfill $\square$

\medskip 

\begin{rem} If $x_0\in S_\infty\cap S_{\vp}$ with $\nabla^2\vp(x_0)$ strictly negative definite,  then there exists $\delta>0$  and $\xi_k\to x_0$ such that $u_k\leq u_k(\xi_k)$ on $B_\delta(x_0)$ and \eqref{19} holds. Then setting $$\eta_k(x)=u_k(\xi_k+r_kx)-u_k(\xi_k),\quad r_k:=e^{-u_k(\xi_k)},$$ one can show that, up to a subsequence, $\eta_k\to \eta$ in $C^{2n-1}_{loc}(\R^n)$, provided $r_k^2\beta_k\to c_0\in[0,\infty)$. The limit function $\eta$ is a solution to \begin{align}\label{rem-eq}(-\D)^n\eta=Q(x_0)e^{2n\eta}\quad\text{in }\R^{2n},\quad \Lambda:=Q(x_0)\int_{\R^{2n}}e^{2n\eta}dx<\infty.\end{align} Then $\eta $ is a spherical solution  if and only if $c_0=0$, and in this case,  we have a quantization of energy around $x_0$. However, if $c_0\neq0$ then $\eta$  is a non-spherical solution, and necessarily $$\lim_{k\to\infty}\int_{B_\delta(x_0)}Q_ke^{2nu_k}dx<\Lambda_1,$$   which  follows from \eqref{rem-39}. 

It is worth pointing out that   non-spherical solutions to \eqref{rem-eq} with $\Lambda\geq\Lambda_1$ (they do exist in dimension $6$ and higher, see \cite{HD, H-volume, HM-gluing, LM-vol}) can not appear as a blow-up limit   if $\nabla\vp(x_0)$ is strictly negative definite.   \end{rem}

\section{Proof of  Theorem \ref{thm2}}

We recall that by  Corollary \ref{cor} we have $S_\infty\cap \{\vp<0\}\subset S_{sph}$.  Therefore, for  every $x_0\in S_\infty$ with  $\vp(x_0)<0$  there exists   $\delta>0$  such that   $B_{2\delta}(x_0)\subset\Omega$,  $B_{2\delta}(x_0)\cap S_\infty=\{x_0\}$ and \eqref{19} holds. 

The following lemma has been proven in \cite{Mar2, Rob2}  (see also \cite{DR}) under the assumption \eqref{cond-uk}. However, here we prove it only  assuming that the blow-up points are not in the zero set of $\vp$. More precisely the following:

\begin{lem}\label{points} Let $x_0\in S_\infty\setminus S_\vp$. Let $\delta>0$ be such that $\bar B_{2\delta}\subset\Omega$ and \eqref{19} holds. Then  there exists an integer $N\geq 1$ and $N$ sequences of points $(x_{i,k})$  with $1\leq i\leq N$ such that, up to a subsequence, the following holds:  
\begin{itemize}
\item[i)] $\lim_{k\to\infty}x_{i,k}= x_0$ for $i=1,\dots,N$. 
\item[ii)] For $1\leq i,j\leq N$ with $i\neq j$ $$\lim_{k\to\infty}\frac{|x_{i,k}-x_{j,k}|}{r_{i,k}}=\infty,\quad r_{i,k}:=e^{-u_k(x_{i,k})}\to0,\quad \frac{u_k(x_{i,k})}{\beta_k}\not\to0.$$
\item[iii)] We have $$\eta_{i,k}(x):=u_k(x_{i,k}+r_{i,k}x)+\log r_{i,k}+\frac{1}{2n}\log\frac{Q_0(x_0)}{(2n-1)!}\to\eta (x)\quad\text{ in }C^{2n-1}_{loc}(\R^{2n})$$ where $\eta$ is a spherical solution to \eqref{eq-R2n}. In particular, $$\lim_{R\to\infty}\lim_{k\to\infty}\int_{B_{Rr_{i,k}}(x_{i,k})}Q_k(x)e^{2nu_k(x)}dx=\Lambda_1.$$
\item[iv)] There exists $C>0$ such that   \begin{align}\inf_{1\leq i\leq N}|x-x_{i,k}|e^{u_k(x)}\leq C\quad\text{for every }x\in B_\delta(x_0).\end{align}
\end{itemize}
\end{lem}
\begin{proof}
We prove the lemma by induction. For $m\geq 1$ we say that $\mathcal{H}_m$ holds if there exists $m$ sequences of points  $(x_{i,k})$ converging to $x_0$ such that, up to a subsequence, the following holds: 
\begin{itemize}
\item[$(\mathcal{H}^1_m)$] for $1\leq i,j\leq m$ with $i\neq j$ $$\lim_{k\to\infty}\frac{|x_{i,k}-x_{j,k}|}{r_{i,k}}=\infty,\quad r_{i,k}:=e^{-u_k(x_{i,k})}\to0,\quad \frac{u_k(x_{i,k})}{\beta_k}\not\to0.$$
\item[$(\mathcal{H}^2_m)$]  for $1\leq i\leq m$ we have  $$\eta_{i,k}(x):=u_k(x_{i,k}+r_{i,k}x)+\log r_{i,k}+\frac{1}{2n}\log\frac{Q_0(x_0)}{(2n-1)!}\to\eta (x)\quad\text{in }C^{2n-1}_{loc}(\R^{2n}),$$ where $\eta$ is a spherical solution to \eqref{eq-R2n}. 
\end{itemize}

Let us first show that $\mathcal{H}_1$ holds. 
Let $x_k=x_{1,k}$ be such that $u_k(x_k)=\max_{ B_{\delta}(x_0)}u_k$. As $u_k\to -\infty$ locally uniformly in $\bar B_\delta(x_0)\setminus\{x_0\}$, we have   $x_k\to x_0$. 
  Splitting the domain  $\Omega=\Omega_1\cup \Omega_2$ where $\Omega_1:=B_{r_k}(x_k)$ and  $\Omega_2:=\Omega\setminus\Omega_1$  with  $r_k:=e^{-u_k(x_k)}$,  from  \eqref{def-vk} we obtain 
\begin{align*}
v_k(x_k)&\leq Ce^{2nu_k(x_k)}\int_{\Omega_1} |\log|x_k-y||dy+C\int_{\Omega_2} |\log|x_k-y||e^{2nu_k(y)}dy \\
&\leq C|\log r_k|\\ &=Cu_k(x_k). 
\end{align*}
This implies that for $k$ large $$C\frac{u_k(x_k)}{\beta_k}\geq \frac{v_k(x_k)}{\beta_k}=\frac{u_k(x_k)}{\beta_k}-\frac{h_k(x_k)}{\beta_k}\geq -\frac12\vp(x_0)>0.$$ Hence,    $\mathcal{H}_1$ follows,  thanks to Lemma \ref{2.6}.

Now we assume that $\mathcal{H}_m$ holds for some $m\geq 1$. We also assume that 
\begin{align}\label{Hm}
\sup_{x\in B_\delta(x_0)}d_{m,k}(x)e^{u_k(x)}\to\infty,\quad d_{m,k}(x):=\inf_{1\leq i\leq m}|x-x_{i,k}| \quad\text{for }x\in B_\delta(x_0).
\end{align}
We claim that $\mathcal{H}_{m+1}$ holds. To prove the claim we  let $x_{m+1,k}$  be given by  $$ d_{m,k}(x_{m+1,k})e^{u_k(x_{m+1,k})}=\sup_{x\in B_\delta(x_0)}d_{m,k}(x)e^{u_k(x)}.$$ Setting  $r_{m+1,k}:=e^{-u_k(x_{m+1,k})}$, by \eqref{Hm}, one has $$\frac{|x_{i,k}-x_{m+1,k}|}{r_{m+1,k}}\xrightarrow{k\to\infty}\infty\quad\text{for }1\leq i\leq m.$$ Moreover, by $(\mathcal{H}^2_m)$ $$\frac{|x_{i,k}-x_{m+1,k}|}{r_{i,k}}\xrightarrow{k\to\infty}\infty\quad\text{for }1\leq i\leq m.$$ Since $$\lim_{k\to\infty}\frac{r_{m+1,k}}{d_{m,k}(x_{m+1,k})}=0,$$  from the definition of $d_{m,k}$ and $r_{m+1,k}$, one gets  for every $R>0$ $$e^{u_k(x)-u_k(x_{m+1,k})}\leq \frac{d_{m,k}(x_{m+1,k})}{d_{m,k}(x)}=1+o(1),\quad\text{for }|x-x_{m+1,k}|\leq R r_{m+1,k},$$ where  $o(1)\to0$ uniformly as $k\to\infty$.  
Using this, as before,  one would  get $\frac{u_k(x_{m+1,k})}{\beta_k}\not\to0$. Thus,  $(\mathcal{H}^1_{m+1})$ holds, and by Lemma \ref{2.6},  $(\mathcal{H}^2_{m+1})$ holds.  This proves our claim.

Since  $(\mathcal{H}^1_m)$ and $(\mathcal{H}^2_m)$ imply that $$\int_{B_\delta(x_0)}Q_ke^{2nu_k}dx\geq m\Lambda_1+o(1),$$ there exists a maximal $m$ such that $\mathcal{H}_m$ holds. Arriving at this maximal $m$, we get that \eqref{Hm} can not hold, and conclude the lemma with $N=m$. 
\end{proof}

A consequence of $iv)$ of Lemma \ref{points} is the following: 
\begin{lem}\label{outside} Let $N$, $\delta$ and $(x_{i,k})$ be as in Lemma \ref{points}. Then there exists $C>0$ such that $$v_k(x)\leq C|\log d_{N,k}(x)|=C\max_{1\leq i\leq N}|\log|x-x_{i,k}||,\quad x\in B_\delta(x_0).$$ In particular, for every  $\rho_k>0$ with $\log \rho_k=o(\beta_k)$, we have $$u_k\to-\infty\quad\text{uniformly in }B_\delta(x_0)\setminus \cup_{i=1}^N B_{\rho_k}(x_{i,k}).$$
\end{lem}
\begin{proof}
 Taking $x_k=x$ and  $\mu_k=d_{N,k}(x)$  in \eqref{log}, and using that $d_{N,k}(x)\leq 2 d_{N,k}(\tilde x)$ for $|x-\tilde x|<\frac12d_{N,k}(x)$, one would get   
 the first part of the lemma, thanks to $iv)$ of Lemma \ref{points}. 
The second part follows immediately from \eqref{19} and  $\vp(x_0)<0$.  
\end{proof}

\medskip   

\noindent\emph{\textbf{Proof of  \eqref{NLambda}}}  From Corollary \ref{cor} we have $ \dot{S}_{sph}:= S_\infty\setminus S_\vp\subset S_{sph}$, and hence, $\dot S_{sph}$ is   either  empty or finite. If the set $\dot S_{sph}$ is empty then \eqref{NLambda} follows trivially as $$u_k\to-\infty\quad\text{locally uniformly in }\Omega\setminus S_\vp,$$ thanks to \eqref{conv-2}. In the later case we denote the set $\dot S_{sph}$ by $\{x^{(1)},\dots, x^{(L)}\}$. We observe that \eqref{NLambda} is equivalent to  \begin{align}\label{xell} \int_{B_{\delta_\ell}(x^{(\ell)})}Q_ke^{2nu_k}dx\to N_\ell\Lambda_1\quad \text{for some }\delta_\ell>0,\,N_\ell\in\mathbb{N},\end{align}   for every $ \ell\in\{1,\dots, L\}.$

For $x_0=x^{(\ell)}\in\dot S_{sph}$ we fix $\delta>0$ such that \eqref{19} holds.  Let $(x_{i,k})$ and $N$ be as in Lemma \ref{points}. For $i\in I:=\{1,\dots,N\}$ we set $$J_i:=\{j\in I :\sup_{k\geq1}\beta_k|x_{i,k}-x_{j,k}|<\infty\}.$$ Note that  $J_{i_1}\cap J_{i_2}=\emptyset$ if and only if $i_2\not\in J_{i_1}$, and $J_{i_1}= J_{i_2}$ if and only if $i_2\in J_{i_1}$. Thus, $I$ can be written as a disjoint union of $J_{i}$'s.  

  We fix $R_i>0$ such that $$\beta_k|x_{i,k}-x_{j,k}|\leq\frac{R_i}{8}\quad\text{for every }j\in J_i.$$
We set $$w_{i,k}(x):=u_k(x_{i,k}+\beta_k^{-1}x)-\log \beta_k\quad \text{for } |x|\leq R_i,\quad i=1,\dots, N.$$ Then, setting $ Q_{i,k}(x):=Q_k(x_{i,k}+\beta_k^{-1}x)$ we have 
\begin{align*}
\left\{ \begin{array}{ll}
(-\D)^n w_{i,k}=Q_{i,k}e^{2n w_{i,k}}\quad\text{in }B_{R_i},\\ \rule{0cm}{.6cm}
\int_{B_{R_i}}e^{2n w_{i,k}}dx\leq C,\\ \rule{0cm}{.6cm}
Q_{i,k}\to Q_0(x_0)>0\quad\text{in }C^0(B_{R_i}).\\
\end{array}
\right.
\end{align*}
In fact, by  $ii)$ of Lemma \ref{points} $$\max_{x\in B_{R_i}}\min_{j\in J_i}|x-\bar x_{i,j,k}|e^{w_{i,k}(x)}\leq C,\quad \bar x_{i,j,k}:=\beta_k(x_{j,k}-x_{i,k}),\quad j\in J_i,$$  
$$w_{i,k}(\bar x_{i,j,k})\to\infty\quad\text{for every }j\in J_i,$$ and for every $j_1\neq j_2$ with $j_1,j_2\in J_i$ $$\frac{|\bar x_{i,j_1,k}-\bar x_{i,j_2,k}|}{\bar r_{i,j_1,k}}\to\infty,\quad \bar r_{i,j_1,k}:=e^{-w_{i,k}(\bar x_{i,j_1,k})}.$$
Moreover, differentiating under the integral sign, from \eqref{def-vk} \begin{align*}\int_{B_{R_i}}|\D w_{i,k}|dx&\leq \beta_k^{-2}\int_{B_{R_i}}\left(|\D v_k(x_{i,k}+\beta_k^{-1}x)| +\beta_k|\D \vp(x_0)|+o(\beta_k)\right)dx\\&\leq o(1)+C\int_{\Omega}e^{2nu_k(y)}\int_{B_{R_i}}\frac{\beta_k^{-2}}{|x_{i,k}-\beta_k^{-1}x-y|^2}dxdy\\&\leq C.\end{align*}
Thefore, by Theorem \ref{RM},  there exists a positive integer $N_i$ (from the proof of Theorem \ref{RM} one would have $N_i=|J_i|$) such that $$\int_{B_\frac{R_i\beta_k}{2}(x_{i,k})}Q_ke^{2nu_k}dx=\int_{B_\frac{R_i}2}Q_{i,k}e^{2nw_{i,k}}dx\xrightarrow{k\to\infty}N_i\Lambda_1.$$  Together with Lemma \ref{outside}  
$$\int_{B_\delta(x_0)}Q_ke^{2nu_k}dx\to N\Lambda_1,\quad N:=\sum_{J_i\text{ disjoint}}N_i.$$ This proves \eqref{xell}. 
\hfill $\square$ 

\medskip  
\medskip 

Under a slightly stronger assumption on $Q_k$, namely $\|Q_k\|_{C^1}\leq C$,   one can have a simpler proof of \eqref{xell} (without using Theorem \ref{RM}). The main idea is to  use a Pohozaev type identity around each peak $x_{i,k}$,   compare  \cite{DR}.  

\begin{lem}\label{gap} Let $x_0=x^{(\ell)}\in\dot S_{sph}$. Let  $\delta>0$ be such that \eqref{19} holds. Let $(x_{i,k})$ and $N$ be as in Lemma \ref{points}. If $(Q_k)$ is bounded in $C^1_{loc}(\Omega)$ then $$|x_{i,k}-x_{j,k}|\beta_k^2\to\infty\quad  \text{for every } i,j\in I:=\{1,\dots,N\} \text{ with } i\neq j,$$ and $$\int_{B_{\beta_k^{-2}}(x_i,k)}Q_ke^{2nu_k}dx\to\Lambda_1.$$ In particular, \eqref{xell} holds with $\delta_\ell=\delta$ and $N_\ell=N$. 
\end{lem}
\begin{proof}
In order to prove the lemma we fix $i\in I$ and consider the set of indices $$J_i:=\{j\in I:\sup_{k\geq 1}\beta_k^2|x_{i,k}-x_{j,k}|<\infty\}.$$ Setting \begin{align*}\rho_{i,k}:=\frac12\min\left\{ \min\{|x_{i,k}-x_{j,k}|:j\in I\setminus J_i \}, \, \beta_k^{-\frac{3}{2}}\right\} 
\end{align*}  we see that $\beta_k^{-2}=o(\rho_{i,k})$.   This would imply 
\begin{align}\frac{(x+y-2x_{i,k})\cdot(x-y)}{|x-y|^2}=-1+o(1)\quad\text{for }(x,y)\in B_{R_i\beta_k^{-2}}(x_{i,k})\times   B^c_{\rho_{i,k}}(x_{i,k}),\label{xy-37}\end{align} where  $R_i>0$ is such that $$|x_{i,k}-x_{j,k}|\leq \frac {R_i}{2}\beta_k^{-2}\quad\text{ for every }j\in J_i.$$ This, and  from the definition of $\rho_{i,k}$     we obtain   \begin{align}u_k\to -\infty \quad \text{uniformly in }\bar B_{\rho_{i,k}}(x_{i,k})\setminus B_{R_i\beta_k^{-2}}(x_{i,k}),\label{xy-38}\end{align} thanks to Lemma \ref{outside}. Moreover,  from the definition of $\rho_{i,k}$ we see that $$\lim_{k\to\infty}\int_{B_{\rho_{i,k}}(x_{i,k})}Q_ke^{2nu_k}dx\geq \Lambda_1.$$ This is a consequence of    $ii)-iii)$ of Lemma \ref{points}.

Next, we  apply  Lemma \ref{poho2} to the integral equation \eqref{barv}.  Indeed, fixing $\psi\in C_c^\infty(B_{2\delta}(x_0))$ with $\psi\equiv1$ on $B_\delta(x_0)$,  from Lemma \ref{poho2} with  $v=\bar v_k$, $K=K_k=\psi Q_ke^{2nh_k}e^{2n\tilde v_k}$, $\xi=x_{i,k}$,   $\tilde\Omega=B_{\rho_{i,k}}(x_{i,k})$ and $$\tilde\lambda_k:=\int_{\tilde\Omega}K_ke^{2n\bar v_k}dx,\quad \lambda_k:=\int_{\Omega}K_ke^{2n\bar v_k}dx,$$ we get \begin{align}&\frac{\tilde \lambda_k}{\Lambda_1}(\lambda_k-\Lambda_1)\notag\\&=-\frac{1}{\Lambda_1} \int_{{\tilde\Omega}}\int_{\Omega\setminus\tilde\Omega}\frac{(x+y-2x_{i,k})\cdot(x-y)}{|x-y|^2}K_k(y)e^{2n\bar v_k(y)}K_k(x)e^{2n\bar v_k(x)}dydx \notag\\ &\quad+\frac{1}{2n}\int_{{\tilde\Omega}}(x-x_{i,k})\cdot \nabla K_k(x)e^{2n\bar v_k(x)}dx-\int_{\partial \tilde\Omega}K_ke^{2n\bar v_k}(x-x_{i,k}) \cdot \nu(x) d\sigma(x)\notag \\ &=:(I)+(II)+(III). \label{39double} \end{align} 
Using  \eqref{xy-37}-\eqref{xy-38} one has $$(I)=\frac{1}{\Lambda_1} (\lambda_k-\tilde\lambda_k)\tilde\lambda_k+o(1).$$ Since $\rho_{i,k}\beta_k\to0$, from \eqref{conv-vp}, \eqref{est-31} and \eqref{xy-38} we infer that $$(II)=o(1)\quad\text{and }(III)=o(1).$$ 
Plugin these estimates in \eqref{39double} we get that  $\tilde\lambda_k\to\Lambda_1$, that is, $$\int_{B_{\rho_{i,k}}(x_{i,k})}Q_ke^{2nu_k}dx\to\Lambda_1,$$ and hence  $J_i=\{i\}$, thanks to   Lemma \ref{points}. 

The lemma follows immediately from Lemma \ref{outside}. 
\end{proof}

\medskip 

Next we prove a stronger version \eqref{uk=betak}. More precisely, we show that if there are multiple spherical bubbles collapsing at a blow-up point $x_0\in\dot S_{sph}$, then the height of each peak has the same order $-\beta_k\vp(x_0)$. \\

\noindent\emph{\textbf{Proof of \eqref{uk=betak}} } Let $x_0=x^{(\ell)}\in \dot S_{sph}$. Let $(x_{i,k})$ be  as in Lemma \ref{points} for some $1\leq i\leq N$. We claim that  for every $1\leq i\leq N$ \begin{align} \label{33} u_k(x_{i,k})=-\beta_k\vp(x_0)(1+o(1)) .\end{align}    

It follows that    $$u_k(x_{k})=u_k(x_{k}+Rr_{k}\sigma)+O(\log R),\quad r_k:=e^{-u_k(x_k)},\quad \text{for }R\geq2,\,\sigma\in S^{2n-1},$$ where we ignored the index $i$ and simply write  $x_k$ for  $x_{i,k}$. For $\ve>0$ we fix $R=R(\ve)>>1$ such that $$\int_{B_{Rr_{k}}(x_{k})}Q_ke^{2nu_k}dx\geq \Lambda_1-\ve\gamma_n.$$ From \eqref{def-vk} 
\begin{align*}
v_k(x_{k}+2Rr_{k}\sigma)&= \frac{1}{\gamma_n}\left(\int_{ B_{Rr_k}(x_k)}+\int_{\Omega\setminus B_{Rr_k}(x_k)}\right)\log\frac{1}{|x_k+2Rr_k\sigma-y|}Q_k(y)e^{2nu_k(y)}dy \\&\geq (2-\ve)u_k(x_k)+O(\log R).
\end{align*}
Splitting $\Omega$ into $$\Omega=\cup_{i=1}^3A_i,\quad A_1:=B_{r_k}(x_{k}+2Rr_k\sigma), \quad A_2:=B_{\beta_k^{-2}}(x_k)\setminus A_1,\quad A_3:=\Omega\setminus(A_1\cup A_2),$$ we write $$v_k(x_{k}+2Rr_{k}\sigma)=\sum_{i=1}^3I_i,\quad I_i:= \frac{1}{\gamma_n}\int_{A_i}\log\frac{1}{|x_k+2Rr_k\sigma-y|}Q_k(y)e^{2nu_k(y)}dy.$$ By $iii)$ of Lemma \ref{points} one has $e^{2nu_k}\leq Cr_k^{-2n}R^{-1}$ on $A_1$. This yields $$I_1\leq C r_k^{-2n}R^{-1}\int_{|y|<r_k}\log\frac{1}{|y|}dy\leq CR^{-1}|\log r_k|\leq C\frac{u_k(x_k)}{R}\leq \ve u_k(x_k),$$ for $R$ sufficiently large.   
From Lemma \ref{gap} $$I_2\leq ( 2+o(1)) u_k(x_k).$$   Using that $\log |x_k+2Rr_k\sigma-y|=O(\log \beta_k)$ on $A_3$ $$I_3\leq C\log\beta_k\int_{A_3}e^{2nu_k}dy=o(\beta_k).$$ Thus  $$ (2-\ve)u_k(x_k)\leq  v_k(x_k+2Rr_k\sigma)+C \leq (2+\ve)u_k(x_k)+o(\beta_k), $$
$$ v_k(x_k+2Rr_k\sigma)=u_k(x_k+2Rr_k\sigma)-h_k(x_k+2Rr_k\sigma)=u_k(x_k)-\beta_k\vp(x_0)+o(\beta_k).$$
The above estimates give  \eqref{33}. 
\hfill $\square$. 

\medskip 

\noindent\emph{\textbf{Proof of \eqref{vol-lower}}} Since \eqref{vol-lower} follows from \eqref{NLambda} if $N_\ell>1$, we only need to consider the case  $N_\ell=1$. 

Let $x_0=x^{(\ell)}\in\dot S_{sph}$ be such that $\nabla^2\vp(x^{(\ell)})>0$. We fix $\delta>0$ such that  $$\nabla^2\vp>0\quad\text{on }B_{2\delta}(x_0),\quad  \vp(x_0)\leq \vp(x)\leq (1-\frac 1n)\vp(x_0)\quad\text{on }B_{2\delta}(x_0),$$  and \eqref{19} holds. We write $u_k=\bar v_k+\bar h_k$ where  (recall that by assumption $Q_k\equiv 1$)$$\bar v_k(x):=\frac{1}{\gamma_n}\int_{B_{\delta}(x_0)}\log\left(\frac{1}{|x-y|}\right)e^{2nu_k(y)}dy,\quad x\in\Omega.$$ We observe that 
$$\|h_k-\bar h_k\|_{C^{2n-1}(B_{\frac32\delta}(x_0))}=\|v_k-\bar v_k\|_{C^{2n-1}(B_{\frac32\delta}(x_0))}\leq C,$$ which gives   $$\frac{\bar h_k}{\beta_k}\to\vp\quad\text{in }C^2(B_{\delta}(x_0)).$$
We  let $\bar h_k(\xi_k):=\min_{B_{\delta }}\bar h_k$. Then $\xi_k\to x_0$ and $\nabla \bar h_k(\xi_k)=0$.  Moreover, one can show that (see the proof of \eqref{est-33})   $$(x-\xi_k)\cdot \nabla \bar h_k(x)\geq c_1\beta_k|x-\xi_k|^2 \quad \text{for }x\in B_{\delta}(x_0),$$ for some $c_1>0$.  Applying Lemma \ref{poho2}   with $v=\bar v_k$, $\Omega=\tilde\Omega=B_{\delta}(x_0)$, $\xi=\xi_k$, $K=e^{2n\bar h_k}$ and $\tilde\lambda=\lambda=\lambda_k:=\int_{B_\delta(x_0)}e^{2nu_k}dx$ we get  \begin{align*} &\frac{\lambda_k}{\Lambda_1}(\lambda_k-\Lambda_1)\\&=\frac{1}{2n}\int_{B_{\delta}(x_0)}(x-\xi_k)\cdot \nabla e^{2n\bar  h_k(x)}e^{2n \bar v_k(x)}dx-\int_{\partial B_\delta (x_0)}(x-\xi_k)\cdot\nu (x)e^{2nu_k(x)}d\sigma\\
&\geq c_1\beta_k\int_{B_\delta (x_0)}|x-\xi_k|^2e^{2nu_k(x)}dx-C\int_{\partial B_\delta (x_0)}e^{2nu_k(x)}d\sigma(x)\\
&=:(I)-(II).
\end{align*}
Using that $u_k(x)=\beta_k(\vp(x)+o(1))\leq \beta_k(1-\frac{1}{2n})\vp (x_0)$  on $\partial B_\delta(x_0)$, and by \eqref{uk=betak} $$(II)\leq C e^{(2n-1)\beta_k\vp(x_0)}=o(1)u_k(x_k)e^{-2u_k(x_k)},$$ where $u_k(x_k):=\max_{ B_{\delta}(x_0)}{u_k}$.  Changing the variable $x\mapsto x_k+r_kx$ with $r_k:=e^{-u_k(x_k)}$ we obtain
\begin{align*}
(I)\geq c_1\beta_k\int_{B_\frac{\delta}{2r_k}(x_k)}g_k(x)e^{2n (u_k(x_k+r_kx)-u_k(x_k))}dx,\quad g_k(x):=|x_k-\xi_k+r_kx|^2.
\end{align*}
 Note that  $g_k(x)\geq r_k^2$ for $|x|\geq 4$ if $|x_k-\xi_k|\leq 2r_k$, and for $|x|\leq 1$ if $|x_k-\xi_k|\geq 2r_k$.  Therefore, by $iii)$ of  Lemma \ref{points} we have $(I)\geq c_2 r_k^2\beta_k$ for some $c_2>0$. 
 
 We conclude \eqref{vol-lower}.
 \hfill $\square$
 
\medskip


In the rest of this section we collect some useful lemmas. 

\begin{lem}\label{2.6}  Let $(u_k)$ be a sequence of solutions to \eqref{eq-1}, \eqref{eq-2} and \eqref{Q0} such that 
 $u_k(x_k)\to\infty$ for some $x_k\to x_0\in\Omega$. Assume that  $\frac{u_k(x_k)}{\beta_k}\not\to0$. Further assume that for every $R>0$  $$u_k(x)\leq u_k(x_k)+o(1)\quad\text{on }B_{Re^{-u_k(x_k)}}(x_k),$$ where $o(1)\to0$ as $k\to\infty$.  
Then  setting  $$\eta_{k}(x):=u_k( x_{k}+r_{k}x)+\log r_{k}+\frac{1}{2n}\log\frac{Q_0(x_0)}{(2n-1)!},\quad r_k:=e^{-u_k(x_k)},$$ we have $$\eta_k\to\eta (x):=\log\left(\frac{2\lambda}{1+\lambda^2| x|^2}\right)\quad\text{in } C^{2n-1}_{loc}(\R^{2n}), \quad \lambda:=\frac12\left(\frac{Q_0(x_0)}{(2n-1)!} \right)^\frac{1}{2n}.$$    In particular  $$\lim_{R\to\infty}\lim_{k\to\infty}\int_{B_{Rr_k}(x_k)}Q_ke^{2nu_k}dx=\Lambda_1.$$
\end{lem}
\begin{proof}
We omit the proof  as it is very similar to that of Step 1 in   Proof of \eqref{conv-2}. The crucial fact $r_k^2\beta_k\to 0$ follows from the hypothesis $\frac{u_k(x_k)}{\beta_k}\not\to 0$.  \end{proof}

The following lemma is a generalization of \cite[Theorem 2.1]{Xu}.
\begin{lem}[Pohozaev Identity]\label{poho2}
Let $v$ be a solution to $$v(x):=\frac{1}{\gamma_n}\int_{\Omega}\log\left(\frac{1}{|x-y|}\right)K(y)e^{2n v(y)}dy,$$ where $K\in C^1(\bar \Omega)$. Then  for  ${\tilde\Omega}\subseteq \Omega$ and $\xi\in\R^{2n}$ we have \begin{align}\label{poho-eq}\frac{\tilde\lambda}{\Lambda_1}(\lambda-\Lambda_1)&=-\frac{1}{\Lambda_1}\int_{{\tilde\Omega}}\int_{\Omega\setminus\tilde\Omega}\frac{(x+y-2\xi)\cdot(x-y)}{|x-y|^2}K(y)e^{2nv(y)}K(x)e^{2nv(x)}dydx \notag\\
&\quad+\frac{1}{2n}\int_{{\tilde\Omega}}(x-\xi)\cdot \nabla K(x)e^{2nv(x)}dx-\int_{\partial \tilde\Omega}Ke^{2nv}(x-\xi) \cdot \nu(x) d\sigma(x)
\end{align}
 where $$\tilde\lambda:=\int_{\tilde\Omega}Ke^{2nv}dx,\quad \lambda:=\int_{\Omega}Ke^{2nv}dx.$$ Moreover, if $K=0$ on $\partial \Omega$ then \begin{align}\label{poho}\int_{\Omega}(\nabla K)e^{2nv}dx=0.\end{align}
\end{lem}  
\begin{proof}
Differentiating under the integral sign we obtain \begin{align}\label{39}\nabla v(x)=-\frac{1}{\gamma_n}\int_{\Omega}\frac{x-y}{|x-y|^2}K(y)e^{2nv(y)}dy.\end{align} 
Multiplying the above identity by $(x-\xi)K(x)e^{2nv(x)}$ and integrating on ${\tilde\Omega}$ 
\begin{align*}
(I):=&\int_{{\tilde\Omega}}(x-\xi)\cdot\nabla v(x)K(x)e^{2nv(x)}dx \\ 
&=-\frac{1}{\gamma_n}\int_{{\tilde\Omega}}\int_{\Omega}\frac{(x-\xi)\cdot(x-y)}{|x-y|^2}K(y)e^{2nv(y)}K(x)e^{2nv(x)}dydx=:(II). 
\end{align*}
Integration by parts yields
\begin{align*}
(I)&=\frac{1}{2n}\int_{{\tilde\Omega}}K(x)(x-\xi)\cdot\nabla e^{2nv(x)}dx\\
&=-\tilde\lambda-\frac{1}{2n}\int_{{\tilde\Omega}}(x-\xi)\cdot \nabla K(x)e^{2nv(x)}dx+\int_{\partial \tilde\Omega}Ke^{2nv}(x-\xi) \cdot \nu(x) d\sigma(x).
\end{align*}
Writing $(x-\xi)=\frac12 (x-y)+\frac12 (x+y-2\xi)$ and using that $\Lambda_1=2\gamma_n$, we compute
\begin{align*}
(II)&=-\frac{\tilde\lambda\lambda}{\Lambda_1}-\frac{1}{\Lambda_1}\int_{{\tilde\Omega}}\int_{\Omega}\frac{(x+y-2\xi)\cdot(x-y)}{|x-y|^2}K(y)e^{2nv(y)}K(x)e^{2nv(x)}dydx\\
&=-\frac{\tilde\lambda\lambda}{\Lambda_1}-\frac{1}{\Lambda_1}\int_{{\tilde\Omega}}\int_{\Omega\setminus\tilde\Omega}\frac{(x+y-2\xi)\cdot(x-y)}{|x-y|^2}K(y)e^{2nv(y)}K(x)e^{2nv(x)}dydx,
\end{align*}
where in the last equality we have used that $\int_{\tilde\Omega}\int_{\tilde\Omega}(\dots)dydx=0$, which follows from the fact that the integrant is antisymmetric in $(x,y)\in \tilde\Omega\times\tilde\Omega$. 
This proves \eqref{poho-eq}.

Multiplying \eqref{39} by $K(x)e^{2nv(x)}$ and integrating on ${\Omega}$ \begin{align*}0&=\frac{1}{\gamma_n}\int_{\Omega}\int_{\Omega}\frac{x-y}{|x-y|^2}K(y)e^{2nv(y)}K(x)e^{2nv(x)}dydx\\ &=-\int_{\Omega}K(x)(\nabla v(x))e^{2nv(x)}dx\\&=\frac{1}{2n}\int_{\Omega}\nabla K(x) e^{2nv(x)}dx,\end{align*} where in the last  equality we used integration by parts, and the first equality follows from the fact that the integrant is antisymmetric in $(x,y)\in\Omega\times\Omega$. 
\end{proof}
 
By a covering argument and by \cite[Lemma 3, Proposition 4]{Mar0} one can prove: 
 \begin{lem}\label{polyharmonic}
Let $\Omega$ be a connected domain in  $\R^n$. Let $\phi_k$ be a sequence of functions on $\Omega$ such that $$\D^m\phi_k=0\quad\text{in }\Omega, \quad \int_{B_R(x_0)}|\phi_k|dx\leq C,\quad \int_{\Omega}\phi_k^+dx\leq C,$$ for some $B_R(x_0)\subset\Omega$. Then, up to a subsequence, we have $$\phi_k\to \phi\quad\text{in }C^\ell_{loc}(\Omega)\quad\text{for every }\ell\in\mathbb N.$$ 
\end{lem}

 \section{The non-local case}\label{section-nonlocal}
 
 Let us first fix some notations, and define the operator $\mathcal L_\frac32$ mentioned in the introduction. 
Points   in $\R^4$ will be denoted by  $X=(x,t)\in\R^3\times\R$.   We will identify $\mathbb R^3=\{(x,t)\,:\,t=0\}=\partial \R^4_+$. In the following, $\bar\Delta$ will denote the Laplacian in $\R^4$ and $\Delta_x$ the Laplacian in $\R^3$.

It is well known that if $U\in W^{2,2}(\mathbb R^4_+)$ is a solution to the problem
\begin{equation}\label{extension0}
\left\{
\begin{split}
&\bar \Delta^2 U=0\quad\text{in }\R^4_+\\
&\partial_t U=0\quad\text{on }\R^3,
\end{split}
\right.
\end{equation}
then $U$ is characterized by the Poisson representation formula
\begin{equation}\label{rep-Poison}
U(x,t)=\frac{4}{\pi^2} \int_{\mathbb R^3}\frac{t^3}{(t^2+|x-\tilde x|^2)^3} u(\tilde x)\,d\tilde x,
\end{equation} with $u=U|_{\R^3}$, see e.g. \cite{CS, Case-Chang, CG,  Chang-Yang, DGHM}. 
Define the operator on $\mathbb R^3$ by
\begin{equation}\label{ope_def}
\mathcal L_{\frac{3}{2}}U:=\frac{1}{2}\lim_{t\rightarrow 0}\partial_t \bar\D U. 
\end{equation}
Then  
\begin{equation*}\label{relation}
\mathcal L_{\frac{3}{2}}U=(-\Delta_x)^{\frac{3}{2}} u,
\end{equation*}
where the $\frac{3}{2}$-fractional Laplacian is defined as the operator with Fourier symbol $|\xi|^{3}$.
Note that $\mathcal L_{\frac{3}{2}}$ can also be defined in a distributional sense. Indeed, given $u\in L_{loc}^1(\mathbb R^3)$, we say that $U\in W^{2,2}(\mathbb R^4_+)$ satisfying $\partial_t U=0$ on $\mathbb R^3$ and $\bar\D^2 U=0$ in $\R^4_+$ is a weak solution to
$\mathcal L_{\frac{3}{2}} U=w, $
 if
\begin{equation*}
0=\int_{\mathbb R^4_+}\bar \Delta U\bar\Delta \psi\,dxdt-2\int_{\mathbb R^3} w\psi \,dx
\end{equation*}
for every test function $\psi\in\mathcal C^{\infty}(\mathbb R^4_+)$ with compact support in $\overline{\mathbb R^4_+}$
and satisfying $\partial_t \psi=0$ on $t=0$.


We say that a solution $U$ to \eqref{extension0} is \emph{representable} if it coincides with its Poisson  representation formula \eqref{rep-Poison} and $\mathcal L_{\frac{3}{2}}U$ is well defined. In particular, the boundary data $u=U|_{t=0}$ of a representable solution  satisfies $$\int_{\R^n}\frac{|u(x)|}{1+|x|^6}\,dx<\infty. $$
Finally we set \begin{equation}\label{K} 
\begin{split}
\K(\Omega):=\big\{ H\in C^\infty(\R^4_+\cup \Omega): \D^2H=0,\,H\le 0, \text{ in }\mathbb R^4_+,
  H\not\equiv 0,\, \partial_tH=\mathcal{L}_\frac32H=0\text{ on }\Omega\big\}.
\end{split}
\end{equation}
 
 We  now  prove a corresponding version  of Lemma \ref{2.6} for the non-local case. 
 
\begin{lem}\label{4.1}  Let $(U_k)$ be a sequence of solutions as in the statement of Theorem \ref{thm-nonlocal}.   Assume that  
 $u_k(x_k)\to\infty$ for some $x_k\to x_0\in\Omega$ such that  $\frac{u_k(x_k)}{\beta_k}\not\to0$. Further assume that for every $R>0$  $$u_k(x)\leq u_k(x_k)+o(1)\quad\text{on }B_{Re^{-u_k(x_k)}}(x_k),$$ where $o(1)\to0$ as $k\to\infty$.  
Then  setting  $$\tilde u_{k}(x):=u_k( x_{k}+r_{k}x)-u_k(x_k),\quad r_k:=e^{-u_k(x_k)},$$ we have \begin{align}\label{tildeu}\tilde u_k\to\tilde u(x):=\log\left(\frac{1}{1+\lambda^2| x|^2}\right)\quad\text{in } C^{2}_{loc}(\R^{3}), \quad \lambda:=\frac12\left(\frac{Q_0(x_0)}{2} \right)^\frac{1}{3}.\end{align}    In particular  $$\lim_{R\to\infty}\lim_{k\to\infty}\int_{B_{Rr_k}(x_k)}Q_ke^{3u_k}dx=\Lambda_1=4\pi^2.$$
\end{lem}
\begin{proof}
It has been shown in \cite{DGHM}  that $U_k$ can be decomposed as $U_k=V_k+H_k$ where \begin{align}\label{def-Vk}
V_k(x,y)=\frac{1}{2\pi^2}\int_{\Omega}\log\left(\frac{1}{|(x,y)-(\tilde x,0)|}\right)Q_k(\tilde x)e^{3u_k(\tilde x)}\,d\tilde x,\quad (x,y)\in \R^3\times\R,
\end{align} and $H_k$ is given by the Poison formula \eqref{rep-Poison} with the boundary data $h_k:=u_k-v_k$, where $v_k:=V_k|_{\R^3}$. It follows that \begin{align} \label{vk-54}  v_k(x)=\frac{1}{2\pi^2}\int_{\Omega}\log\frac{1}{|x-\tilde x|}Q_k(\tilde x)e^{3u_k(\tilde x)}d\tilde x. \end{align} The function $H_k$ can be extended on $\R^4_-$ by setting $H_k(x,t)=H_k(x,-t)$. Then $H_k$ satisfies $$\lim_{t\to0}\frac{\partial}{\partial_t}H_k(x,t)=0=\lim_{t\to0}\frac{\partial}{\partial_t}\bar \D H_k(x,t)\quad\text{for every }x\in\Omega.$$ Therefore, $$\bar \D^2 H_k=0\quad\text{on }\R^4\setminus(\Omega^c\times\{0\}).$$ In fact, for every bounded open set ${\bf \Omega}\subset\R^4$ satisfying ${\bf \Omega}\cap\R^3\Subset\Omega$, we have  \begin{align}\label{conv-55}  \frac{H_k}{\beta_k}\to\Phi\quad\text{in }C^\ell(\bar{\bf\Omega}),\end{align} for every integer $\ell\geq0$. 
 
Now we set $$\tilde u_k(x)=u_k(x_k+r_kx)-u_k(x_k),\quad \tilde U_k(X)=U_k((x_k,0)+r_kX) -U_k((x_k,0)),$$ $$\tilde v_k(x)=v_k(x_k+r_kx)-v_k(x_k),\quad \tilde V_k(X)=V_k((x_k,0)+r_kX) -V_k((x_k,0)),$$
$$ \tilde h_k(x)=h_k(x_k+r_k x)-h_k(x_k)  ,\quad   \tilde H_k(X)=H_k((x_k,0)+r_kX) -H_k((x_k,0)) .$$  Using  the uniform bounds \eqref{cond-vol} and \eqref{extra-assumption} one can show that $$\int_{\R^3}\frac{ |\tilde v_k(x)|}{1+|x|^{3+\ve}}dx\leq C(\ve),\quad \int_{\Omega_k}\tilde u_k^+(x)dx\leq C,\quad \int_{\R^3\setminus\Omega_k}\frac{\tilde u_k^+(x)}{1+|x|^6}dx\to0,$$ for every $\ve>0$, where $\Omega_k:=\{x\in\R^3:x_k+r_kx\in\Omega\}$.  Hence,  $$ \int_{\Omega_k}\tilde h_k^+(x)dx   +\int_{\R^3\setminus\Omega_k}\frac{\tilde h_k^+(x)}{1+|x|^6}dx\leq C.$$ This  and  the Poision representation formula of $\tilde H_k$  leads to   $$\limsup_{k\to\infty}\int_{\mathbb B_R} \tilde H_k^+dX\leq CR^{4-\ve}\quad\text{for every }\mathbb B_R\subset\R^4.$$  Since $r_k^2\beta_k\to0$, we have that $$\bar \D\tilde H_k\to0,$$ thanks to \eqref{conv-55}. Therefore, up to a subsequence, $\tilde H_k\to\tilde H$  in $C^4_{loc}(\R^4)$ for some   harmonic  function $\tilde H$ satisfying $$\int_{\mathbb B_R}\tilde H^+dX\leq CR^{4-\ve}.$$ This implies that $\tilde H $ is bounded from above, and hence $\tilde H\equiv 0$. 

Next we show that the sequence $(\tilde v_k)$ is bounded in $C^2_{loc}(\R^3)$. Indeed, setting $\tilde Q_k(x):=Q_k(x_k+r_k x)$ we see that $$\tilde v_k(x)=\frac{1}{2\pi^2}\int_{\Omega_k}\log\frac{|y|}{|x-y|}\tilde Q_k(y)e^{3\tilde u_k(y)}dy,\quad \int_{\Omega_k}\tilde Q_ke^{3\tilde u_k}dy\leq C.$$ As the function $\tilde u_k$ satisfies $$\tilde u_k\leq o(1)\quad\text{on }B_R,\quad\text{for every fixed }R>0,$$ differentiating under the integral sign, one easily gets that the sequence  $(|\nabla\tilde v_k|)$ and $(|\nabla^2\tilde v_k|)$ are bounded in $C^0_{loc}(\R^3)$. Therefore, as $\tilde v_k(0)=0$, we conclude that $(\tilde v_k)$ is bounded in $C^2_{loc}(\R^3)$. 

Thus, up to a subsequence, $\tilde u_k=\tilde v_k+\tilde h_k\to \tilde v=:\tilde u$ in $C^1_{loc}(\R^3)$. It easily follows that $\tilde u$ satisfies $$\tilde u(x)=\frac{1}{2\pi^2}\int_{\R^3}\log\frac{|y|}{|x-y|}\tilde Q_0(x_0)e^{3\tilde u(y)}dy,\quad \tilde u\leq \tilde u(0),\quad \quad \int_{\R^3}e^{3\tilde u}dy<\infty.$$
Then a classification result to the above integral equation in \cite{Xu}   imply that $\tilde u$ should be of the form  given in \eqref{tildeu}.
   \end{proof}
   
   \medskip 

\noindent{\bf\emph{Proof of Theorem \ref{thm-nonlocal}} }  The proof of \eqref{conv-23} is very similar to the one of \eqref{conv-2}. Here one needs to use Lemma \ref{4.1} and the  representation formula \eqref{def-Vk}. 

The quantization result   \eqref{quant-24} can be proved using a Pohozaev type identity for the integral equation \eqref{vk-54}. Notice that  corresponding versions of  Lemmas \ref{points}, \ref{outside} and \ref{gap} for the non-local case follow easily.  
\hfill $\square$

\section{Examples}\label{Examples}

In  Example 1 we show that the convergence  in \eqref{conv-2} is sharp in the sense that $C^{1,\alpha}_{loc}(\Omega\setminus S_{sph})$ can not be replaced by  $C^{2}_{loc}(\Omega\setminus S_{sph})$.   \medskip

\noindent\textbf{Example 1}  
Let  $u$  be an entire solution to  \eqref{eq-R2n} with $m=2n\geq 4$ such that (see \cite{CC,hm, WY} for existence of such solutions) $$u(x)=-|x|^2+O(\log|x|)\quad \text{as }|x|\to\infty.$$ Then $u$ satisfies the integral equation \begin{align}\label{35}u(x)=\frac{1}{\gamma_n}\int_{\R^{2n}}\log\left( \frac{|y|}{|x-y|}\right)e^{2nu(y)}dy-|x|^2+c,\end{align} for some $c\in\R$.   Differentiating under the integral sign, from \eqref{35} \begin{align*}\D u(0)=-\frac{2n-2}{\gamma_n}\int_{\R^{2n}}\frac{1}{|y|^2}e^{2nu(y)}dy-2n< -2n.\end{align*} We set $$u_k(x)=u(kx)+\log k\quad \text{for }x\in B_1.$$ Then $u_k$ satisfies \eqref{eq-1}-\eqref{eq-2} with $Q_k\equiv (2n-1)!$, $\Omega =B_1$ and for some $\Lambda>0$. It follows from \eqref{35} that $$\frac{u_k}{\beta_k}\to \vp\quad\text{in }C^{2n-1}_{loc}(B_1\setminus\{0\}),\quad \vp:=-|x|^2,\quad \beta_k:=k^2.$$ Notice that $$\frac{\D u_k(0)}{\beta_k}=\D u(0)<-2n=\D \vp(0).$$ Therefore, $\frac{u_k}{\beta_k}\not\to\vp$ in $C^2_{loc}(B_1)$ (here $S_{sph}=\emptyset$).

\medskip 

\noindent\textbf{Example 2} It has been shown in \cite{H-volume} (see also \cite{HD, LM-vol}) that for every $n\geq 3$ and  $\Lambda>\Lambda_1$   there exists a radially symmetric solution $u$ to \begin{align}\label{36}(-\D )^nu=e^{2nu}\quad \text{in }\R^{2n},\quad \int_{\R^{2n}}e^{2nu}dx=\Lambda.\end{align} In fact $u$ is give by $$u(x)=\frac{1}{\gamma_n}\int_{\R^{2n}}\log\left( \frac{|y|}{|x-y|}\right)e^{2nu(y)}dy+c_1|x|^2-c_2|x|^4+c_3,$$  for some $c_1,c_2>0$ and $c_3\in\R$. We set $$u_k(x)=u(kx)+\log k,\quad x\in B_1.$$ Then  $u_k$ satisfies \eqref{conv-1} with  $  \vp:=-c_2|x|^4,\,\beta_k:=k^4$ and $$\lim_{\ve\to0}\lim_{k\to\infty}\int_{B_\ve}e^{2nu_k}dx=\Lambda.$$

\medskip 

Next example shows that one can have $S_{sph}\cap S_\vp\neq \emptyset$.  \medskip 

\noindent\textbf{Example 3} Let $n= 3$ and let $\Lambda>\Lambda_1$ be fixed. Then  there exists a sequence of radially symmetric solutions $(u_k)$ to \eqref{36} such that (see \cite{HM-gluing})$$u_k(0)\to\infty,\quad u_k(1)\to\infty , \quad \int_{B_2}e^{2nu_k}dx\to\Lambda,$$ and $u_k$ is given by $$u_k(x)=\frac{1}{\gamma_n}\int_{\R^{2n}}\log\left( \frac{1}{|x-y|}\right)e^{2nu(y)}dy-(1+o(1)) u_k(0)(1-|x|^2)^2+c_k,$$  $$c_k\to\infty \quad \text{and }c_k=o(u_k(0)).$$ Notice that $u_k|_{B_2}$ satisfies \eqref{conv-1} with $$\beta_k:=  u_k(0),\quad \vp(x)=-(1-|x|^2)^2,\quad S_{sph}=\{0\}.$$ Let $\rho_k\to\infty$ slowly enough so that $$e^{-2u_k(0)}u_k(0)\rho_k^2\to0.  $$ Set $\bar u_k(x)=u_k(\rho_kx)+\log\rho_k$ on $B_1$.  Then $\bar u_k$ satisfies \eqref{conv-1} with  $$ \bar\vp:= -|x|^4,\quad \bar\beta_k:=u_k(0)\rho_k^4 .$$ Moreover,  as $e^{-2\bar u_k(0)}\bar \beta_k\to0$, we have  $0\in S_{sph}$, and hence $S_{sph}=S_{\bar \vp}=\{0\}$.  Also note that a quantization result does not hold for $\bar u_k$ as $$\lim_{\ve\to0}\lim_{k\to\infty}\int_{B_\ve}e^{2n\bar u_k}dx=\Lambda.$$

\medskip   

\noindent\textbf{Acknowledgements}  The author is greatly thankful to Luca Martinazzi and Pierre-Damien Thizy   for various stimulating discussions.

\end{document}